\documentclass[a4paper]{article}
\usepackage[all]{xy}
\usepackage[dvips]{graphics,graphicx}
\usepackage{amsfonts,amssymb,amsmath,color,mathrsfs, amstext}
\usepackage{amsbsy, amsopn, amscd, amsxtra, amsthm,authblk}
\usepackage{enumerate,algorithmic,algorithm}
\usepackage{upref}
\usepackage{geometry}
\usepackage[displaymath]{lineno}
\usepackage{float}
\usepackage{yhmath}
\usepackage{booktabs}
\usepackage{subcaption}
\usepackage{multirow}
\usepackage{makecell}
\usepackage[colorlinks,
            linkcolor=red,
            anchorcolor=red,
            citecolor=red
            ]{hyperref}

\numberwithin{equation}{section}

\def\R{\mathbb{R}}
\def\E{\mathbb{E}}

\def\cP{\mathcal{P}}

\def\cH{\mathcal{H}}
\def\cF{\mathcal{F}}

\newcommand{\intdx}[1]{\int_{\mathbb{R}^{dN}} #1\,\mathrm{d} x^N}
\newcommand{\intdX}[1]{\int_{\mathbb{R}^{dN}} #1\,\mathrm{d} X^N}

\newtheorem{theorem}{Theorem}[section]
\newtheorem{lemma}{Lemma}[section]

\newtheorem{remark}{Remark}[section]
\newtheorem{corollary}{Corollary}[section]

\title{Solving stationary nonlinear Fokker-Planck equations via sampling}

\author[a]{Lei Li\thanks{E-mail: leili2010@sjtu.edu.cn}}
\author[b]{Yijia Tang\thanks{E-mail: yijia\_tang@sjtu.edu.cn}}
\author[b]{Jingtong Zhang\thanks{E-mail: sffred@sjtu.edu.cn}}

\affil[a]{School of Mathematical Sciences, Institute of Natural Sciences, MOE-LSC, Shanghai Jiao Tong University, Shanghai, 200240, P.R.China.}
\affil[b]{School of Mathematical Sciences, Shanghai Jiao Tong University, Shanghai, 200240, P.R.China.}

\date{}

\begin{document}
\maketitle

\begin{abstract}
Solving the stationary nonlinear Fokker-Planck equations is important in applications and examples include the Poisson-Boltzmann equation and the two layer neural networks.  Making use of the connection between the interacting particle systems and the nonlinear Fokker-Planck equations, we propose to solve the stationary solution by sampling from the $N$-body Gibbs distribution. This avoids simulation of the $N$-body system for long time and more importantly such a method can avoid the requirement of uniform propagation of chaos from direct simulation of the particle systems. We establish the convergence of the Gibbs measure to the stationary solution when the interaction kernel is bounded (not necessarily continuous) and the temperature is not very small. Numerical experiments are performed for the Poisson-Boltzmann equations and the two-layer neural networks to validate the method and the theory.
\end{abstract}

\section{Introduction}

The nonlinear Fokker-Planck equations describe the evolution of laws for the McKean-Vlasov processes \cite{mckean1967,funaki1984certain}, which often describe the evolution of distributions under the self-consistent mean field \cite{jabin2014}. A famous example is the Poisson-Nernst-Planck (PNP) model describing the motion of charged chemical species in a fluid medium, which is important is electrochemistry \cite{brumleve1978numerical,flavell2014conservative}. In this paper, we are concerned with solving the stationary solutions of the nonlinear Fokker-Planck equations. For one species, dynamic nonlinear Fokker-Planck equation may be written as
\begin{gather}\label{eq:nonlinearFP}
\partial_t\rho=\nabla\cdot(\rho(\nabla U+\nabla W*\rho))+\beta^{-1}\Delta\rho,
\end{gather}
and thus the stationary equation we are concerned is given by
\begin{gather}\label{eq:stanfp}
\nabla\cdot(\rho(\nabla U+\nabla W*\rho))+\beta^{-1}\Delta\rho=0.
\end{gather}
Here, $U$ is some external potential field, $W$ is the interacting potential.
The multi-species case can be similarly written out. For example, if there are two species, the stationary equations could be written as 
\begin{gather}\label{eq:stanfp2species}
\begin{split}
&\nabla\cdot(\rho_1 (\nabla U_1+\nabla W_1*\rho_1+\nabla W_c*\rho_2))+\beta^{-1}\Delta\rho_1=0,\\
& \nabla\cdot(\rho_2 (\nabla U_2+\nabla W_2*\rho_2
+\nabla W_c*\rho_1 ))+\beta^{-1}\Delta\rho_2=0.
\end{split}
\end{gather}
Here, $U_i$'s are the external potentials for the two species, $W_i$'s are the interaction potentials within the species and $W_c$ is the cross species interaction, which are assumed to be symmetric. See section \ref{sec:multispecies} for the details.

The nonlinear Fokker-Planck equations are naturally associated with the interacting particle systems at the microscopic level, which are ubiquitous in applications such as molecular dynamics \cite{frenkel2001understanding}, flocking and swarming \cite{cucker2007emergent,degond2017coagulation}, chemotaxis \cite{horstmann2003}, clustering \cite{holland1983} and consensus \cite{motsch2014}. The models consisting of exchangeable particles (one species) may be described by the first order interacting particle systems, i.e., for $i=1,\cdots, N$
\begin{equation}\label{eqn:SDE}
d X_i=-\nabla U(X_i) d t-\frac{1}{N-1}\sum_{j=1, j\neq i}^N\nabla W(X_i-X_j)d t+\sqrt{\frac{2}{\beta}}\,d B_i.
\end{equation}
The particle systems for multi-species could be similarly written out, and see section \ref{sec:multispecies} for details.  Here, we will call the label $X_i$ the ``position". The concrete significance may be other quantities (for example, $X_i$ represents the opinion in the opinion dynamics and may represent the velocity in the homogeneous Landau equation).  The processes  $B_i\in \mathbb{R}^d$ are $N$ independent standard Brownian motions. The initial data $X_i^0$'s are i.i.d. sampled from some distributions. 
There are many models that are described by the second order interacting particle systems if the ``velocity" is considered (see section \ref{subsec:seconddiscussion}). In some applications, the first order systems arise as the zero-inertia/overdamped limit of the second order Langevin systems \eqref{eqn:SDE2nd}. As remarked in section \ref{subsec:seconddiscussion}, when we consider the stationary distributions, the second order systems are not very special. In this work, we will talk about first order systems as the main examples to be focused.

The factor $1/(N-1)$ indicates that we are focusing on the regime (by choosing suitable units for physical quantities) where the interaction energy is comparable to the external energy and the total mass of the particles is of order $1$. This scaling is crucial for the connection between the particle system \eqref{eqn:SDE} and the nonlinear Fokker-Planck equation \eqref{eq:nonlinearFP}.
In this regime, as $N\to\infty$, one particle is expected to feel a mean field given by $-\nabla U-\nabla W*\rho$ and the motions of two given particles move independently.
That the motions of some focused particles tend to be independent copies under the mean field is called the ``propagation of chaos". The term ``propagation" means that the chaotic configuration is kept from the initial time $t=0$. This is the classical mean field limit and this regime may be called the mean field regime.  The nonlinear Fokker-Planck equations \eqref{eq:nonlinearFP} thus describe the mean field limits of the interacting particle systems.  The rigorous justification of the mean field limit and the propagation of chaos could date back to Kac and McKean \cite{kac1956,mckean1967} and see \cite{dobrushin1979vlasov,sznitman1991} for some classical works. Currently, this is still a popular research topic, especially for singular potentials \cite{serfaty2020mean,jabin2017,bresch2019mean}. Most of these results are for finite time $T$. In various practical applications, such as the Poisson-Boltzmann (PB) equation \cite{guoy1910constitution,chapman1913li} and the two-layer neural network \cite{mei2018mean}, we are more interested in the stationary solution of the nonlinear Fokker-Planck equations. In fact, the PB equation describes the stationary form of the PNP model. Classical discretization using finite difference or finite element may not be convenient for irregular domain and suffers from curse of dimensionality. A popular approach is to do simulation of the interacting particle systems and run the time long enough so that the solution could be approximation of the stationary solution of the nonlinear Fokker-Planck equations \cite{li2022some}. This requires the simulation for long time to reach equilibrium. Moreover, theoretically, this needs the uniform propagation of chaos, which needs a lot of additional requirements \cite{guillin2019uniform,guillin2021uniform}. These results often need the potentials to be convex in some sense \cite{guillin2019uniform} or to be considered on torus \cite{guillin2021uniform} and thus are limited in applications. Hence, there is a gap between the computation and the theory.

We are then motivated to adopt another approach, namely we choose not to do simulation of the $N$-particle systems directly, and instead to do sampling from the stationary distribution, i.e., the Gibbs distribution of the particle systems. One efficient sampling method from the $N$-body Gibbs distribution could be the Random Batch Monte Carlo method \cite{li2020random}.  We then use the samples as the approximation to the solution of the stationary solutions. Moreover, the justification of the convergence from the stationary Gibbs distribution to the nonlinear Fokker-Planck equations may avoid the strong requirements of the uniform propagation of chaos and the result here is valid in $\R^d$.
Provided that $W$ is bounded and continuous, using the standard large deviation principle, the empirical measure converges to minimizer of mean field free energy but there is not explicit convergence rate given (see, for example, the discussion in \cite{chafai2014first}). When $W$ is not continuous, the large deviation argument fails.

In this paper, we will use the tool of relative entropy to gauge the discrepancy of the joint law and the tensor product of $N$ copies of the stationary solution, inspired by the recent works for propagation of chaos \cite{jabin2016mean,jabin2018quantitative,serfaty2020mean}. Then,  the convergence of the marginal and the empirical measures can be established and an explicit rate could be given. In particular, we prove that if $\beta$ is not very big (temperature is not small), both the marginal distribution and empirical measure for the $N$-body Gibbs measure converge to a stationary distribution of the nonlinear Fokker-Planck equation. Moreover, we also establish the results for multiple species cases. The results can thus be used to justify solving the stationary nonlinear Fokker-Planck equations using sampling algorithms like the random batch Monte Carlo method \cite{li2020random}.

The rest of the paper is organized as follows. Section \ref{sec:stationary} is devoted to the introduction of the basic properties of the $N$-body Gibbs measure for the particle system \eqref{eqn:SDE} and the stationary solution to nonlinear Fokker-Planck equation \eqref{eq:nonlinearFP}. In Section \ref{sec:main}, we show the convergence of the $N$-body Gibbs measure to the stationary solution of the nonlinear Fokker-Planck equation for one species, which illustrates our main methodology and lays the foundation of sampling for solving the nonlinear Fokker-Planck equations. In Section \ref{sec:multispecies}, we consider the multi-species cases. The convergence is proved for the two-species case as the typical example.  In Section \ref{sec:numer}, we perform numerical experiments for two typical applications, namely the  PB equations and the two layer neural networks, where the sampling from the $N$-body Gibbs distribution is performed by the random batch Monte Carlo method.

\section{Setup and properties of the Gibbs measure and stationary solutions}
\label{sec:stationary}

In this section, we first collect some basic facts and properties for the Gibbs measure of the interacting particle systems in section \ref{subsec:ips} and the stationary solution of the nonlinear Fokker-Planck equation \eqref{eq:nonlinearFP}, namely \eqref{eq:stanfp} in section \ref{subsec:nonfp}. For the clarity of the presentation, we focus only on the one species case here. The multi-species case will be studied in section \ref{sec:multispecies}. Then, the idea of sampling to solve the stationary nonlinear Fokker-Planck equations is explained in section \ref{subsec:sampling}. Discussion on the second order systems will be performed in section \ref{subsec:seconddiscussion}.

\subsection{The Gibbs measure of the interacting particle system}\label{subsec:ips}

The interacting particle system \eqref{eqn:SDE} is associated with the following energy functional
\begin{gather}\label{eq:NbodyHam}
\begin{split}
E_N(x_1,\cdots,x_N)&=\sum_{i=1}^N U(x_i)+\frac{1}{(N-1)}\sum_{1\leq i<j\leq N }W(x_i-x_j)\\
&=\sum_{i=1}^N U(x_i)+\frac{1}{2(N-1)}\sum_{i\neq j}W(x_i-x_j).
\end{split}
\end{gather}
Note that we do not call this ``Hamiltonian'' as called in other literature because this is a functional of the spatial variables only without the conjugate variables (i.e., the momentum). Then, \eqref{eqn:SDE} can be written as 
\begin{gather}\label{eq:gradientformfirst}
dX_i=-\nabla_{X_i} E_N\, dt+\sqrt{\frac{2}{\beta}}\,dB_i,
\end{gather}
or in terms of $X:=(X_1, \cdots, X_N)$ and $\boldsymbol{B}=(B_1,\cdots, B_N)$
\begin{gather}
dX=-\nabla_X E_N\,dt+\sqrt{2/\beta}\,d\boldsymbol{B}.
\end{gather}

\begin{remark}\label{rmk:rescaledHam}
If we think that each particle has weight $1/N$, as considered in \cite{chafai2014first} the energy may be given by
\begin{gather}
\tilde{E}_N(x_1,\cdots,x_N)=\frac{1}{N}\sum_{i=1}^N U(x_i)+\frac{1}{N(N-1)}\sum_{1\leq i<j\leq N }W(x_i-x_j).
\end{gather}
With this scaling, the (scaled) inverse temperature $\beta$ should be like $\beta=N\bar{\beta}$ to be consistent with \eqref{eq:gradientformfirst} (the time should be rescaled as well).
Another different scaling has been used in \cite{leble2017large}, where the energy used is roughly $NE_N=N^2\tilde{E}_N$.
\end{remark}

The joint distribution of $(X_1,\cdots, X_N)$ is known to satisfy the following $N$-body Fokker-Planck (FP) equation \cite{bogachev2022fokker}
\begin{gather}\label{eq:nbodyfp}
\partial_t\rho_N=\nabla_{\mathrm{x}}\cdot(\nabla_{\mathrm{x}} E_N\,\rho_N)+\beta^{-1}\Delta_{\mathrm{x}}\rho_N,
\end{gather}
where the derivative is taken with respect to $\mathrm{x}:=(x_1,\cdots,x_N)$.
This linear Fokker-Planck is also connected to the following free energy for the $N$-body system
\begin{gather}\label{eq:nbodyfree}
\mathcal{F}_N(\rho_N):=\int_{\R^{Nd}} E_N \rho_N d\mathrm{x}+\beta^{-1}\int_{\R^{Nd}} \rho_N\log \rho_N\,d\mathrm{x}.
\end{gather}
Physically, this is the interaction energy plus the entropy. The equation \eqref{eq:nbodyfp}
is the gradient flow of \eqref{eq:nbodyfree} under the Wasserstein distance.
We recall that the $W_2$ gradient flow of the free energy functional $F$ is given by \cite{jordan1998variational}:
\begin{gather}
\partial_t \rho = \nabla \cdot \left(\rho\nabla \frac{\delta F}{\delta \rho}\right).
\end{gather}
Based on this, the following is standard. 
\begin{lemma}[Gibbs distribution]
\label{lem:steadystate}
Suppose that $\exp(-\beta E_N)\in L^1(\R^{Nd})$. Then, the $N$-body system \eqref{eqn:SDE} has a unique invariant probability measure, which is the N-body Gibbs measure given by
\begin{equation}\label{eqn:Gibbsnbody}
\begin{split}
\rho_N &=\bar{Z}_N^{-1} \exp(-\beta E_N)\\
&= \bar{Z}_N^{-1}\exp\left(-\beta\left(\sum_i U(x_i)+\frac{1}{2(N-1)}\sum_{i\neq j}W(x_i-x_j)\right)\right),
\end{split}
\end{equation}
where $\bar{Z}_N$ is the normalizing factor such that the integral is $1$.
This is the unique minimizer of the free energy \eqref{eq:nbodyfree}.
\end{lemma}

It is clear that any (local) minimizer of $\cF_N$ is a stationary solution of \eqref{eq:nbodyfp} (but not vice versa).  It is easy to verify that \eqref{eqn:Gibbsnbody} is a stationary solution of \eqref{eq:nbodyfp}.  Any stationary distribution $\tilde{\rho}_N$ must have a full support in $\R^{Nd}$ and be absolutely continuous with respect to the Lebesgue measure \cite[Corollary 3.5.8]{bogachev2022fokker}.
The uniqueness of the stationary solution follows then from \cite[Theorem 4.1.11]{bogachev2022fokker}. This then implies that the free energy has a unique minimizer, which is clear because the free energy functional $\cF_N$ is convex with respect to the linear structure.

\subsection{The stationary solution of the nonlinear Fokker-Planck equation}\label{subsec:nonfp}

In this subsection, we investigate the stationary solution of the nonlinear Fokker-Planck equations. 
Similar to the $N$-particle case, the nonlinear Fokker-Planck equation \eqref{eq:nonlinearFP} is associated with the mean-field free energy 
\begin{equation} \label{eq:meanfieldfree}
  \cF(\rho):=\int_{\mathbb{R}^d} U \rho\ \mathrm{d}x 
+\frac{1}{2} \int_{\mathbb{R}^d} \rho W\ast\rho \ \mathrm{d}x +\beta^{-1}\int_{\mathbb{R}^d} \rho\log \rho \ \mathrm{d}x.
\end{equation}
In fact, it is the gradient flow of $\cF$ under the $W_2$ distance. 
As well-known, the mean-field free energy may exhibit phase transitions and could multiple local minimizers, and the global minimizers may not also be unique \cite{dawson1983critical}.

For any local minimizer, taking variation subject to $\rho \ge 0$ and $\int \rho \,dx=1$, one has
\begin{gather}
\frac{\delta \cF}{\delta \rho} = U+ W \ast\rho + \beta^{-1} (\log \rho +1)=\text{const}
\end{gather}
on the support of $\rho$. One can further verify that the support of the minimizer is full (see the appendix of \cite{mei2018mean}). These results can also be derived formally using the KKT conditions \cite{chen2022fluctuation}. 
Hence, any minimizers of $\cF$ satisfies the following relation
\begin{equation}\label{eq:rho}
\rho=Z^{-1}\exp(-\beta(U+W\ast\rho)),
\end{equation}
where $Z$ is the normalizing constant, and can be verified to be a stationary solution of  \eqref{eq:nonlinearFP} and thus a solution to \eqref{eq:stanfp}.

\subsection{Sampling to solve the stationary nonlinear Fokker-Planck equations}\label{subsec:sampling}

We aim to connect the Gibbs measure \eqref{eqn:Gibbsnbody} with the stationary solution \eqref{eq:stanfp} so that we can then do sampling from
\eqref{eqn:Gibbsnbody} to solve the stationary Fokker-Planck equations. 
Our goal in this work is then to identify the convergence of \eqref{eqn:Gibbsnbody} to one of the stationary solution \eqref{eq:stanfp} in the form of \eqref{eq:rho} in certain sense. 

To make the meaning precise, consider the $k$-marginal of \(\rho_N\) defined as
\begin{gather}
\rho_{N,k}(x_1,\dots,x_k)=\int_{\R^{(N-k)d}}\rho_N(x_1,\dots, x_k, dx_{k+1}, \cdots, dx_N)
\end{gather}
and the empirical distribution 
\begin{gather}
\mu_N:=\frac{1}{N}\sum_{i=1}^N\delta(\cdot-X_i),
\quad (X_1, X_2,\cdots, X_N)\sim \rho_N.
\end{gather}
We aim to show that $\rho_{N,1}$ and $\mu_N$ converge to $\rho$ in certain sense.  As can be imagined, when there is phase transition so that the local minimizers of $\cF$ are not unique while the $N$-body system has a unique Gibbs measure, there could be intrinsic difficulty. Our result is that when the temperature is large enough, the stationary solution of the form \eqref{eq:rho} is unique and the convergence can be quantified. The approach in section \ref{sec:main} is to use the relative entropy to do the quantitative estimate.

In particular, using the relative entropy, one can conclude that $\rho_{N,1}$ converges to $\rho$ in total variation norm with rate $1/\sqrt{N}$ and $\mu_N$ converges to $\rho$ in $H^{-\alpha}$, $\alpha>d/2$ in expectation. See section \ref{sec:main} and \ref{sec:multispecies} for details.  
In practical applications, we combine these two to obtain the following algorithm.
\begin{algorithm}[H]
\caption{Sampling to solve the nonlinear Fokker-Planck equation}
\begin{algorithmic}[1]
\STATE Set $N\ge 2$ and $N_s\ge 1$.
\STATE Draw $N_s$ samples from $\rho_N$ using certain sampling methods (e.g., some Markov chain Monte Carlo methods like RBMC).
\STATE Collect all $X_i$'s from all the $N_s$ samples to obtain $N\times N_s$ points in $\R^d$.
\STATE Use the $NN_s$ points to form a density for approximation of $\rho$ using 
\begin{gather}
\bar{\mu}_N=\frac{1}{NN_s}\sum_{j=1}^{NN_s}\delta(\cdot-X_j).
\end{gather}
\end{algorithmic}
\label{alg:samplingfornonfp}
\end{algorithm}
Clearly, $\bar{\mu}_N$ is the average of $N_s$ empirical measures.
One may do some post processing to make this distribution more smooth.
Clearly,
\[
\|\bar{\mu}_N-\rho\|_{H^{-\alpha}}\le \frac{1}{N_s}\sum_{m=1}^{N_s}
\|\mu_N^{(m)}-\rho\|_{H^{-\alpha}}.
\]
Here, $\mu_N^{(m)}$ means the empirical measure for the $m$th sample.
Since the typical rate of convergence for $\mu_N$ to $\rho$ is $N^{-1/2}$,  this is also $N^{-1/2}$. In applications, sometimes we care about the convergence of some  quantity of the form
$\int \varphi(\cdot)d\bar{\mu}_N$. This is again the weak convergence and we expect that the rate is again $N^{-1/2}$ like the law of large numbers. 

We remark also that the algorithm and the proof later in section \ref{sec:main} and \ref{sec:multispecies}  can be applied to cases when $x_i$ is in bounded domains (like torus and bounded domain with reflection boundary condition). See section  \ref{subsec:pb} for one example.

\subsection{Discussion on the second order interacting particle systems}\label{subsec:seconddiscussion}

If the ``velocity" is considered, the models for the interacting particle systems are described by the second order systems.  For example, for the Langevin system, the particle system may be given by the following for $i=1,\cdots, N$
\begin{equation}\label{eqn:SDE2nd}
\begin{split}
& d X_i= V_i\,dt,\\
& dV_i=-\gamma V_i\,dt-\nabla U(X_i) d t-\frac{1}{N-1}\sum_{j=1, j\neq i}^N\nabla W(X_i-X_j)d t+\sqrt{\frac{2\gamma}{\beta}}\,d B_i.
\end{split}
\end{equation}
This particle system is also associated with the energy functional \eqref{eq:NbodyHam}.
In fact, it can be written as
\begin{gather}\label{eq:2ndgradient}
\begin{split}
& d X= V\,dt,\\
& dV=-\gamma V\,dt-\nabla_{X} E_N\,dt+\sqrt{\frac{2\gamma}{\beta}}\,d\boldsymbol{B}.
\end{split}
\end{gather}

\begin{remark}
If we consider the rescaled energy functional in Remark \ref{rmk:rescaledHam}, and each particle has mass $1/N$, one may obtain
\[
\begin{split}
& d X= V\,dt,\\
& \frac{1}{N}dV=-\tilde{\gamma} V\,dt-\frac{1}{N}\nabla_{X} E_N\,dt+\sqrt{\frac{2\tilde{\gamma}}{\tilde{\beta}}}\,d\boldsymbol{B}.
\end{split}
\]
Now, we need $\tilde{\beta}=N\beta$ and $\tilde{\gamma}N=\gamma$ to be consistent with \eqref{eqn:SDE2nd}. The regime considered is the one where the total mass is $1$, the temperature is low, and the friction coefficient is also small. 
\end{remark}

Using \eqref{eq:2ndgradient}, it is not difficult to verify that the invariant measure is given by
\[
\pi(d\mathrm{x}, d\mathrm{v})\propto \exp\left(-\beta \left(E_N+\frac{|\mathrm{v}|^2}{2}\right)\right)d\mathrm{x}d\mathrm{v}.
\]
The Gibbs measure for the second order systems is just the Gibbs measure of the first order system tensored by the Maxwellian distribution in the velocity space. From the sampling viewpoint, we only have to focus on the Gibbs measure of the first order systems.  Similarly, the second order nonlinear Fokker-Planck equations may be written as
\begin{gather}\label{eq:kineticnonfp}
\partial_t f=-\nabla_x\cdot(vf)+\nabla_v\cdot(f(\gamma v+\nabla U+\nabla W*_x f))
+\beta^{-1}\gamma \Delta_vf,
\end{gather}
where $*_x$ means the convolution in the spatial variable.  One can find that if $\rho(x)$ is a stationary solution to \eqref{eq:stanfp}, then $\rho(x)e^{-v^2/2}$ is a solution to \eqref{eq:kineticnonfp}. Hence, if we choose to solve stationary equations of this form by sampling, we do not have to distinguish the first order and second order systems. 

In some applications, there may be coupling between the space and velocity variables. Then, the invariant measure is no longer of the product form. Nevertheless, one may view $(X_i, V_i)$ as a new particle, and the convergence of the Gibbs measure to the solution of the stationary nonlinear Fokker-Planck equation can be similarly studied as in section \ref{sec:main} or \ref{sec:multispecies}.

\section{Convergence of the stationary solutions for one species}
\label{sec:main}

In this section, we study the convergence of the Gibbs measure directly using the relative entropy as mentioned in the introduction \cite{jabin2016mean,jabin2018quantitative,lim2020quantitative}. This may allow weaker assumptions on $W$. Compared to many results in literature \cite{jabin2016mean,jabin2018quantitative,lim2020quantitative,wang2021gaussian}, our result is valid in $\R^d$.

To start with, we recall some basic notations. Let $\cP(E)$ denote the set of probability measures on a generic Polish space $E$. The relative entropy (KL divergence) between two probability measures $\mu$ and $\nu$ on Polish space $E$ is defined as
\[
\mathcal{H}(\mu\mid \nu) = \left\{
\begin{aligned}
  &\int_E \log \frac{\mathrm{d} \mu}{\mathrm{d}\nu} \mathrm{d}\mu, & \text{if} ~ \mu \ll \nu,\\
 & \infty, & \text{else},
\end{aligned}\right.
\]
where $\frac{\mathrm{d} \mu}{\mathrm{d}\nu}$ denotes the Radon-Nikodym derivative of $\mu$ with respect to $\nu$. 
Let $\mathcal{P}_{\text{sym}}(E^N)$ denote the set of symmetric probability measures on the product space $E^N$, that is for any Borel measurable set $A\subset E^N$ and any permutation $\sigma$ of $\{1,\cdots,N\}$, $\mu_N(A)=\mu_N(\sigma(A))$. For two probability measures $\mu, \nu \in \mathcal{P}(E^{k})$, the scaled (normalized) relative entropy is given by
\begin{gather}
\mathcal{H}_k(\mu\mid \nu)=\frac{1}{k}\mathcal{H}(\mu\mid \nu).
\end{gather}
The scaled relative entropy satisfies the following useful monotonicity property \cite{hauray2014,jabin2018quantitative}.
\begin{lemma}\label{lem:monotonicity}
For any probability density \(\rho_N\in \mathcal{P}_{\text{sym}}(E^{d})\) and \(\rho\in \mathcal{P}(E)\). 
Then,
\begin{equation*}
    \mathcal{H}_k(\rho_{N,k}\mid\rho^{\otimes k})\leq \mathcal{H}_N(\rho_N\mid\rho^{\otimes N}),\quad  1\leq k\leq N.
\end{equation*}
Here $\rho_{N,k}\in \mathcal{P}(E^{k})$ is the $k$-marginal of \(\rho_N\) defined as
\[
\rho_{N,k}(x_1,\dots,x_k)=\int_{E^{(N-k)}}\rho_N(x_1,\dots, x_k, dx_{k+1}, \cdots, dx_N),
\]
while $\rho^{\otimes k} \in \mathcal{P}(E^{k})$ is given by 
$\rho^{\otimes k}(x_1,\cdots, x_k)=\prod_{i=1}^k{\rho(x_i)}$.
\end{lemma}

Next, we state our main result. 
\begin{theorem}\label{thm:main}
Suppose $V$ is an external potential such that $\exp(-\beta H_N)\in L^1(\R^{Nd})$,  and the interacting kernel $W(\cdot)\in L^{\infty}(\mathbb{R}^{d})$. Then, if $\beta<(2e\sqrt{2}\|W\|_{\infty})^{-1}$, for a minimizer $\rho$ of the free energy $\cF$, there exists $C$ independent of $N$ and $N_0>0$, s.t.
\begin{gather}
\mathcal{H}_N(\rho_{N}\mid \rho^{\otimes N})\leq \frac{C}{N}, \quad \forall N>N_0.
\end{gather}
\end{theorem}

This result indicates that $\mathcal{H}_N(\rho_{N}\mid \rho^{\otimes N})\to 0$ as $N \to \infty$ under the assumptions of Theorem \ref{thm:main}. Due to the indistinguishability assumption on the particles, their joint law is invariant under relabelling of the particles. That is, $\rho_N$ together with $\rho_{N,k}$ are symmetric. It follows from Lemma \ref{lem:monotonicity} that $\rho_{N,k}\rightharpoonup \rho^{\otimes k}$, which is indeed the propagation of chaos at equilibrium.
Specifically, the first marginal of the invariant Gibbs measure (the law of a single particle) weakly converges to the limit law, i.e. $\rho_{N,1}\rightharpoonup \rho$ as $N \to \infty$.   Here, the conditions are weaker than those for uniform propagation of chaos.

To prove Theorem \ref{thm:main}, we need some auxiliary lemmas. The first is a type of Fenchel-Young's inequality and has been written out explicitly in \cite[Lemma 1]{jabin2018quantitative}.
\begin{lemma}
\label{lem:integralcontrol}
For any two probability densities $\rho$ and $\tilde{\rho}$ on $E$ and some test function $\Phi\in L^1(\rho)$, one has that $\forall \eta>0$,
\begin{equation*}
\int_{E}\Phi \rho(dx)\leq\frac{1}{\eta}\left(\mathcal{H}(\rho\mid\tilde{\rho})+\log{\int_{E} e^{\eta\Phi}\,\tilde{\rho}(dx)}\right).
\end{equation*}
\end{lemma}

Next, we need the following lemma.
\begin{lemma}\label{lmm:expcontrol}
Consider the random variable 
\begin{gather}
Z=(W\ast \rho) (X_1) 
-W(X_1-X_2)+(W\ast \rho) (X_2)-\int_{\mathbb{R}^d} \rho W\ast \rho \ \mathrm{d}x,
\end{gather}
where $X_1, X_2$ are i.i.d. from $\rho$.
If there exists $\alpha>0$ such that $\lim_{p\to\infty}\alpha^{-p}\|Z\|_{L^p}^p\le 1$, then for $\eta<\sqrt{2}/(e\alpha)$, one has
\begin{multline}\label{eq:expcontrol}
\limsup_{N\to\infty} \int\exp\Big(\eta\sum_{i=1}^N\Big(W*\rho(x_i)  \\
-\frac{1}{2(N-1)}\sum_{j:j\neq i}W(x_i-x_j)-\frac{1}{2}\int W*\rho \rho \,dx\Big)\Big) \rho^{\otimes N}\,dx_1\cdots dx_N<\infty.
\end{multline}
\end{lemma}
The result above is reminiscent of the ``large deviation estimate'' by Jabin and Wang {\cite[Therorem 4]{jabin2018quantitative}} using combinatoric techniques. 
\begin{lemma}\label{lem:largedeviation}
For any probability measure $\rho$ on $E$, $\{X_1,X_2.\cdots\}$ be a sequence of i.i.d. random variables with distribution $\rho$. $\phi(x,y)$ satisfies $\gamma:=C(\sup_{p\ge 1}\|\sup_z|\phi(\cdot, z)|\|_{L^p(\rho)}/p)^2<1$ where $C$ is a universal constant. Assume that $\phi$ satisfies the following cancellations
\[
\int_{E} \phi(x,y)\rho(x)\mathrm{d}x =0 \quad \forall y,\quad \int_{E} \phi(x,y)\rho(y)\mathrm{d}y =0 \quad \forall x.
\]
Then
\[
\sup_{N\geq 2} \int_{E^N} \exp\left(\frac{1}{N} \sum_{i,j=1,i\neq j}^N \phi(X_i,X_j)\right)\rho^{\otimes N} \leq \frac{2}{1-\gamma}<\infty.
\]
\end{lemma}
It is tempted to apply Lemma \ref{lem:largedeviation} with 
\[
\phi(x,y) = (W\ast \rho) (x) -\frac{1}{2}W(x-y)-\frac{1}{2}\int_{\mathbb{R}^d}\rho W\ast \rho\,\mathrm{d}x.
\]
It is easily found that $\int_{\mathbb{R}^d} \phi(x,y)\rho(x)\mathrm{d}x = \int_{\mathbb{R}^d}\rho  (W\ast \rho)\ \mathrm{d}x - (W\ast\rho) (y)$
and $\int_{\mathbb{R}^d} f(x,y)\rho(y)\mathrm{d}y = (W\ast\rho) (x)-\int_{\mathbb{R}^d}\rho  (W\ast\rho)\,\mathrm{d}x$
do not cancel out.

Our key observation in this work is that if we introduce a symmetric version of the above function
\begin{equation}\label{eqn:phi}
\phi(x,y) = \frac{1}{2}(W\ast\rho) (x) + \frac{1}{2}(W\ast\rho) (y)
-\frac{1}{2}W(x-y)-\frac{1}{2}\int_{\mathbb{R}^d}\rho (W\ast\rho)\,\mathrm{d}x,
\end{equation}
then $\phi$ satisfies the two cancellation conditions in Lemma \ref{lem:largedeviation}
and Lemma \ref{lmm:expcontrol} can be established.
Here, with this observation, we would like to accomodate the martingale proof in  \cite[Lemma 4.3]{lim2020quantitative} to the new case here. We remark that the
results in \cite[Lemma 4.3]{lim2020quantitative} is weaker than Lemma \ref{lem:largedeviation} but the conditions needed could be explicit in the proof and might be generalized to more other cases in the future.

\begin{proof}[{Proof of Lemma \ref{lmm:expcontrol}}]
Let $X_1, \cdots, X_N$ be i.i.d. samples from $\rho$.
Let $\mathcal{F}_k=\sigma(X_1,\cdots, X_k)$.

Define
\[
D_k=\sum_{j=1}^{k-1} \Big((W\ast \rho) (X_j) 
-W(X_j-X_k)\Big)+ (k-1)(W\ast \rho) (X_k)-(k-1)\int_{\mathbb{R}^d} \rho W\ast \rho \ \mathrm{d}x.
\]
It can be verified that $\E(D_k | \cF_{k-1})=0$
and
\[
\sum_{k=1}^N D_k=(N-1)\sum_{k=1}^NW*\rho(X_k)
-\sum_{k,j: k>j}W(X_j-X_k)-\frac{N(N-1)}{2}\int_{\mathbb{R}^d} \rho W\ast \rho \ \mathrm{d}x.
\]
The estimates \eqref{eq:expcontrol} then becomes $\E\exp(\frac{\eta}{N-1}\sum_{k=1}^ND_k)$.

Applying the Marcinkiewicz-Zygmund type inequality (see \cite[Lemma 4.3]{lim2020quantitative}), one has
\[
\| \sum_{k=1}^ND_k\|_{L^p}^2 \le (p-1)\sum_{k=1}^N \|D_k\|_{L^p}^2.
\]
Define
\[
B_j^k=(W\ast \rho) (X_j) 
-W(X_j-X_k)+(W\ast \rho) (X_k)-\int_{\mathbb{R}^d} \rho W\ast \rho \ \mathrm{d}x.
\]
Since $\E(B_j^k | X_k)=0$, one can apply the Marcinkiewicz-Zygmund type inequality
inequality again to obtain that 
\[
\|D_k\|_{L^p}^2\le (p-1)\sum_{k=1}^{k-1} \|B_j^k\|_{L^p}^2.
\]
Consequently,
\[
\| \sum_{k=1}^ND_k\|_{L^p}\le (p-1)\sqrt{\sum_{k=1}^N \sum_{j=1}^{k-1} \|B_j^k\|_{L^p}^2 }=(p-1)\sqrt{\frac{N(N-1)}{2}}\|B_1^2\|_{L^p}.
\]
Using the this estimate, one has
\[
\E\exp\left(\frac{\eta}{N-1}\sum_{k=1}^ND_k\right)
\le 1+\sum_{p\ge 2}\frac{\eta^p (p-1)^p \|B_1^2\|_{L^p}^p}{p!}\left(\frac{N}{2(N-1)}\right)^p,
\]
and the result follows by Stirling's formula.
\end{proof}

Now, we now prove  Theorem \ref{thm:main}.
\begin{proof}[Proof of Theorem \ref{thm:main}]

First, one has
\[
\prod_{i=1}^N\rho(x_i)
=Z^{-N}\exp(-\beta(\sum_{i=1}^N U(x_i)+\sum_{i=1}^N W*\rho(X_i))).
\]
According to Lemma \ref{lem:steadystate}, the Gibbs measure can be rewritten as
$$
 \rho_N=Z_N^{-1}\exp\left(-\beta\left(\sum_i U(x_i)+\frac{1}{2(N-1)}\sum_{i\neq j}W(x_i-x_j)+\frac{N}{2}\int_{\mathbb{R}^d}\rho (W\ast \rho)\,\mathrm{d}x \right)\right),
$$
where $Z_N=\bar{Z}_N\exp(-\frac{N\beta }{2}\int_{\mathbb{R}^d}\rho(W\ast \rho)\mathrm{d}x)$.  
Let $\rho$ be a stationary solution of the nonlinear Fokker-Planck equation.  Direct computation reveals that
\[
\begin{aligned}
    &\mathcal{H}(\rho^{\otimes N}\mid\rho_N)
    = \ N \int_{\mathbb{R}^d}\rho\log \rho\mathrm{d}x - \intdx{\rho^{\otimes N} \log \rho_N }\\
=& -N \log Z+\log Z_N. 
\end{aligned}
\]
Then, the nonnegativity of the relative entropy yields the relation
\begin{equation}
    -N \log Z+\log Z_N\geq 0.
\end{equation}

On the other hand,
\[
\begin{aligned}
    &\mathcal{H}(\rho_N\mid\rho^{\otimes N})
    =\intdx{ \rho_N \log \rho_N }-
    \intdx{ \rho_N \log \rho^{\otimes N} }\\
    =& \intdx{ \rho_N \left(-\log Z_N-\beta\left(\sum_{i=1}^N U(x_i)+\frac{1}{2(N-1)}\sum_{i,j: i\neq j}W(x_i-x_j)+\frac{N}{2}\int_{\mathbb{R}^d}\rho(W\ast \rho)\right)\right)}\\
    &- \intdx{\rho_N\left(-N\log Z -\beta\left(\sum_{i=1}^N U(x_i)+\sum_{i=1}^N W\ast\rho(x_i)\right)\right)}\\
    =&\ N \log Z - \log Z_N + \beta\intdx{\rho_N \Phi }
    \le \beta \intdx{\rho_N \Phi },
\end{aligned}
\]
where 
\begin{equation}\label{eq:nbodyPhi}
   \Phi (x_1,\cdots,x_N) =\left(-\frac{1}{2(N-1)}\sum_{i,j: i\neq j}W(x_i-x_j) +\sum_{i=1}^N W\ast\rho(x_i)-\frac{N}{2}\int_{\mathbb{R}^d}\rho (W\ast \rho)\right).
\end{equation}

Obviously we do not know the properties of $\rho_N$ and would much prefer having expectations with respect to the tensorized law $\rho^{\otimes N}$. 
Applying Lemma \ref{lem:integralcontrol} with $\eta=c\beta$, one has
\begin{equation}
\label{eqn:expectation}
   \intdx{\Phi\rho_N}\leq\frac{1}{c\beta}\left(\mathcal{H}(\rho_N\mid\rho^{\otimes N})+\log\intdx{\exp(c\beta \Phi)\rho^{\otimes N}}\right). 
\end{equation}
Hence, for \(c>1\), we can get
\begin{equation*}
\mathcal{H}(\rho_N\mid\rho^{\otimes N})\leq \frac{1}{c-1}\log \intdx{\exp(c\beta \Phi)\rho^{\otimes N}}.
\end{equation*}

Applying Lemma \ref{lmm:expcontrol}, since $\|Z\|_{L^{\infty}}\leq 4\|W\|_{L^{\infty}}<\infty$, we find that when 
\[
\beta<\frac{1}{2\sqrt{2}e\|W\|_{\infty}},
\]
there exists $c>1$ and $N_0>0$ such that
\[
\mathcal{H}(\rho_N\mid\rho^{\otimes N}) \le \frac{C}{c-1}, \forall N>N_0.
\]
\end{proof}

As a direct consequence of Theorem \ref{thm:main}, the minimizer of the mean-field free energy $F$ is unique when $\beta$ is small enough. In other word, there exists a unique stationary solution $\rho$ to the nonlinear Fokker-Planck equation. 

Besides, we can obtain the convergence in the total variation norm or the Wasserstein metric under the same condition. The results follow from the monotonicity property, the Csisz\'ar–Kullback–Pinsker inequality \cite{villani2009optimal} for the TV norm and  the Talagrand transportation inequality for the Wasserstein distance \cite{bobkov1999exponential}.
\begin{corollary}
\label{cor: converge rate}  
For $\beta<(2e\sqrt{2}\|W\|_{\infty})^{-1}$ and $N$ large enough, for any fixed $k$,
\[
\|\rho_{N,k}-\rho^{\otimes k}\|_{TV}\leq \sqrt{2\mathcal{H}(\rho_{N,k} \mid \rho^{\otimes k})}\leq \sqrt{\frac{2 k}{N}\mathcal{H}(\rho_{N} \mid \rho^{\otimes N})}\lesssim
\sqrt{\frac{k}{N}}.
\]
Similar results hold for the Wasserstein distance
\[
W_1(\rho_{N,1},\rho) \lesssim \sqrt{\frac{k}{N}}.
\]
\end{corollary}

The following estimate concerns on the convergence from the empirical measure $\mu_N=\frac{1}{N}\sum_{i=1}^N\delta (x-X_i)$ of particle system \eqref{eqn:SDE} at steady state to $\rho$ in $H^{-\alpha}$, for $\alpha > d/2$. Together with the tightness, this can imply the weak convergence of the empirical measures, justifying using sampling to compute the stationary nonlinear Fokker-Planck equations.
\begin{corollary}
\label{cor:sobolevestimate}
For each $\alpha > d/2$, there exists constant $C$ depends only on $\alpha$ such that
\begin{equation*}
    \mathbb{E}\|\mu_N-\rho\|_{H^{-\alpha}}^2\leq \frac{C}{N}
\end{equation*}
for $\beta$ small enough and $N$ large enough.
\end{corollary}
\begin{proof}[Proof of Corollary \ref{cor:sobolevestimate}]
The proof is the whole space version of {\cite[Lemma 2.6]{wang2021gaussian}}. 

Since the Dirac measure belongs to $H^{-\alpha}(\mathbb{R}^d) $ for $\alpha > d/2$, $\mu_N - \rho \in H^{-\alpha}(\mathbb{R}^d)$.
Denote $\Psi = \|\mu_N-\rho\|_{H^{-\alpha}}^2$, then
\[
\Psi = \int_{\mathbb{R}^d}\left(1+|\xi|^2\right)^{-\alpha}|\widehat{(\mu_N - \rho) }(\xi)|^2\mathrm{d}\xi.
\]
Here, $\hat{f}$ represents the Fourier transform of $f$.
Since  $\Psi <\infty$, applying Lemma \ref{lem:integralcontrol} with $\eta\leftarrow \eta N$
\[     
\mathbb{E} \Psi = \intdx{\Psi\rho_N}\leq\frac{1}{\eta}\left(\mathcal{H}_N(\rho_N\mid\rho^{\otimes N})+\frac{1}{N}\log\intdx{\exp(N\eta\Psi)\rho^{\otimes N}}\right).
\]
Let $C_{\alpha}=\int_{\mathbb{R}^d} \frac{1}{(1+|\xi|^2)^{\alpha}}\mathrm{d}\xi$, then the Jensen's inequality gives
\[
\exp(N\eta\Psi) \leq \int_{\mathbb{R}^d} \frac{ \left(1+|\xi|^2\right)^{-\alpha}}{C_{\alpha}}  \exp\left(N\eta C_{\alpha}|\widehat{(\mu_N - \rho)}(\xi)|^2\mathrm{d}\xi\right).
\]
Since 
\[
|\widehat{(\mu_N - \rho)}(\xi)|^2
=\frac{1}{N^2}\sum_{i,j}\mathrm{Re}\left[\left(e^{i\xi \cdot X_i}-\int_{\mathbb{R}^d} e^{i\xi\cdot x}\rho(x)\mathrm{d}x \right) \left(e^{-i\xi \cdot X_j}-\int_{\mathbb{R}^d} e^{-i\xi\cdot y}\rho(y)\mathrm{d}y \right)\right].
\]
Then, Lemma \ref{lem:largedeviation} implies that for $\eta$ small enough (depending on $\alpha$ and $W$), $\exists \ C$ depends on $\alpha$, s.t.
$$
\sup_{\xi}\sup_{N\geq 2} \intdX{\exp\left(N \eta C_{\alpha}\langle \phi,\mu_N\otimes \mu_N\rangle
\right)\rho^{\otimes N}} \leq C.
$$
Together with Theorem \ref{thm:main}, one then concludes that
\[
 \mathbb{E}\|\mu_N-\rho\|_{H^{-\alpha}}^2\leq \frac{C}{N}.
 \]
\end{proof}

Again, we do not need $W$ and $U$ be to continuous any more and we have the rate of convergence. The rate $1/\sqrt{N}$ is widely considered to be optimal for the convergence of empirical measures as it corresponds to the stochastic fluctuations.

As a next corollary, we consider the fluctuation measure
\[
\mathcal{G}_N=\sqrt{N}(\mu_N-\rho)
\]
around the mean field limit. Corollary
\ref{cor:sobolevestimate} suggests 
\begin{equation} \label{eqn:hypo}
\limsup_{N\to\infty} \mathbb{E}\|\mathcal{G}_N \|_{H^{-\alpha}}^2 \leq C \text{~for~} \alpha > d/2,    
\end{equation}
when $\beta$ is small enough.
Here, we consider the interaction fluctuation where $W$ does not necessarily belong to $H^{\alpha}$ with $\alpha>d/2$.
\begin{corollary}
\label{cor:boundness}
For $\beta<(2e\sqrt{2}\|W\|_{\infty})^{-1}$ there exists $N_0>0$ such that
\begin{equation}\label{eqn:hypo2}
  \sup_{N>N_0} \mathbb{E}\langle W(x-y), \mathcal{G}_N^{\otimes 2} \rangle <\infty.  
\end{equation}
\end{corollary}

\begin{proof}
It can be computed that 
\[
 \langle W(x-y), \mathcal{G}_N^{\otimes 2} \rangle 
 = \frac{1}{N}\sum_{i,j=1}^NW(X_i-X_j)-2\sum_{i=1}^N W*\rho(X_i)+N\int_{\mathbb{R}^d}\rho (W\ast \rho)\,\mathrm{d}x.
 \]
This is essentially $-\frac{1}{2}\Phi$ with $\Phi$ given in \eqref{eq:nbodyPhi}. The only difference is that 
we have $1/N$ here while there is $1/(N-1)$ there. The difference is clearly tiny. 
Using the same proof for Theorem \ref{thm:main}, the result follows.
\end{proof}

In \cite{delgadino2021diffusive}, they also tried to study the fluctuation of the Gibbs measure. There, they assumed \eqref{eqn:hypo2} while we can see that \eqref{eqn:hypo2} actually holds when the interaction kernel is bounded (not necessarily continuous).

\section{Convergence for the multi-species cases}\label{sec:multispecies}

In this section, we consider the multi-species case, which is actually common in applications. For example, in electrochemistry, we often want to find the distributions of different charged chemical species in a fluid medium. To make the presentation clean, we consider only two species, while the general case can be treated similarly.

\subsection{The problem setup}

We will assume that there are $N$ particles for each species. 
The particles for the first species are labelled by $X_i$, $1\le i\le N$
while the particles for the second species are labelled by $Y_j$, $1\le j\le N$.
The equations of the motion using the overdamped Langevin may be given by
\begin{equation}\label{eq:twospecies}
\begin{aligned}
dX_i = -\left(\nabla U_1(X_i) + \frac{1}{N-1}\sum_{k=1, k\neq i}^N{\nabla W_1(X_i- X_j)}+ \frac{1}{N}\sum_{j=1}^N{\nabla  W_c(X_i-Y_j)}\right)dt&\\
 + \sqrt{\frac{2}{\beta}}\,dB_{1i}(t), \quad  i=1, \dots, N, &\\
dY_j = -\left(\nabla U_2(Y_j) + \frac{1}{N-1}\sum_{\ell=1, \ell \neq j}^N{\nabla W_2(Y_j-Y_\ell)}- \frac{1}{N}\sum_{i=1}^N{\nabla W_c(X_i-Y_j)}\right)dt &\\
+ \sqrt{\frac{2}{\beta}}\,dB_{2j}(t), \quad j=1, \dots, N.&
\end{aligned}
\end{equation}
Here, $U_1, U_2$ are the external potentials for the two species, $W_1$ is the interaction potential within the first species, $W_2$ is the interaction potential within the second species and $W_c$ is the cross species interaction. The Brownian motions $B_{1,i}, B_{2,j}$ are all independent.
For the charged chemical species, one may have $W_1(z)=z_1^2\Phi(z)$, $W_2(z)=z_2^2\Phi(z)$ and $W_c(x, y)=z_1z_2\Phi(x-y)$, where $z_i$'s are the valence of the ions and $\Phi$ being the Coulomb potential. Typically, $z_1=1$ and $z_2=-1$. We assume that all the interaction potentials are symmetric while the asymmetric case can be similarly treated. Then, one has
\[
\nabla W_c(x-y)=-\nabla W_c(y-x).
\]

It is not hard to find the corresponding nonlinear Fokker-Planck equations are given by
\begin{gather}\label{eq:nonlinearFP2species}
\begin{split}
&\partial_t\rho_1=\nabla\cdot(\rho_1 (\nabla U_1+\nabla W_1*\rho_1+\nabla W_c*\rho_2))+\beta^{-1}\Delta\rho_1,\\
& \partial_t\rho_2=\nabla\cdot(\rho_2 (\nabla U_2+\nabla W_2*\rho_2
+\nabla W_c*\rho_1 ))+\beta^{-1}\Delta\rho_2.
\end{split}
\end{gather}
The stationary equations are thus \eqref{eq:stanfp2species}.

The energy functional for the system \eqref{eq:twospecies} is 
\begin{multline}
E_N(\mathrm{x},\mathrm{y}) = \sum_{i=1}^N{U_1(x_i)} + \sum_{i=1}^N{U_2(y_i)} + \frac{1}{2(N-1)}\sum_{i,j=1, i\neq j}^N{W_1(x_i, x_j)} \\
+ \frac{1}{2(N-1)}\sum_{i,j=1, i\neq j}^N{W_2(y_i, y_j)} + \frac{1}{N}\sum_{i,j=1}^N{W_c(x_i, y_j)},
\end{multline}
where $\mathrm{x}=(x_1,\cdots, x_N)$
and $\mathrm{y}=(y_1,\cdots, y_N)$.

Similar to the one species case, the equations \eqref{eq:twospecies} are actually
\begin{equation}
d\begin{pmatrix}X\\Y\end{pmatrix} = -\nabla_{(\mathrm{x},\mathrm{y})} E_N(X,Y)dt + \sqrt{\frac{2}{\beta}}\,d\boldsymbol{B},
\end{equation}
so that the Gibbs distribution is given by
\begin{gather}\label{eq:multirhoN}
\rho_N(x_1,\dots,x_N,y_1,\dots,y_N)=\bar{Z}_N^{-1}\exp(-\beta E_N(\mathrm{x},\mathrm{y})).
\end{gather}

The nonlinear Fokker-Planck equation has the free energy
\begin{multline}
\mathcal{F}(\rho_1, \rho_2)=\frac{1}{2}\int_{\R^d\times\R^d} W_1(x-x')\rho_1(x)\rho_1(x')dxdx'+\frac{1}{2}\int_{\R^d\times\R^d} W_2(y-y')\rho_2(y)\rho_2(y')dydy' \\
+\int_{\R^d\times \R^d} W_c(x-y)\rho_1(x)\rho_2(y)dxdy+
\beta^{-1}\int \rho_1\log\rho_1\,dx+\beta^{-1}\int \rho_2\log\rho_2\,dy.
\end{multline}
The equation can be viewed as the $W_2$ gradient flow in the space $\cP(\R^d)\times \cP(\R^d)$, namely
\[
\partial_t\rho_1=\nabla\cdot(\rho_1\nabla \frac{\delta \cF}{\delta\rho_1}),\quad \partial_t\rho_2=\nabla\cdot(\rho_2\nabla \frac{\delta \cF}{\delta\rho_2}).
\]
By similar argument, a stationary pair $(\rho_1, \rho_2)$ satisfies 
\begin{equation}
\begin{aligned}
\rho_1(x) = & Z_1^{-1}\exp\Big(-\beta\left(U_1(x)+W_1*\rho_1(x)+W_c*\rho_2(x)\right)\Big), \\
\rho_2(y) = & Z_2^{-1}\exp\Big(-\beta\left(U_2(y)+W_2*\rho_2(y)+W_c*\rho_1(y)\right)\Big),
\end{aligned}
\end{equation}
where $Z_i$'s are normalizing constants.

Our approach to investigate the convergence is observe that the Gibbs distribution  is symmetric about $(X_i, Y_i)$ and thus we will look at the convergence to the 
joint distribution $\rho_1(x)\rho_2(y)dxdy\in \cP(\R^d\times \R^d)$.
Using this idea, one may view $Z_i\equiv (X_i, Y_i)$ as a grand particle.
The interaction energy $\frac{1}{N}W_c(X_i, Y_i)$ can then be treated as the self-interaction. This formally can be reduced to the one species case. 
This viewpoint, however, cannot identify the independence between $x$ and $y$ in the mean field limit explicitly. 

\begin{remark}
Such a treatment clearly can only deal with the case when the particle numbers are the same for all species. If the numbers are different, such a treatment would not work.
\end{remark}

\subsection{The convergence result and the sketch of the proof}

We will use the treatment in the above section to obtain the quantitative convergence. 
In particular, we have the following claim.

\begin{theorem}\label{thm:convmulti}
Suppose that the interaction potentials are bounded and  
\begin{gather}\label{eq:betamultispecies}
\beta 2\sqrt{2}e(2\|W_c\|_{\infty}+\max(\|W_1\|_{\infty}, \|W_2\|_{\infty}))<1,
\end{gather}
Let $(\rho_1, \rho_2)$ be a stationary pair and define $\rho_1\rho_2(dx, dy):=\rho_1(x)\rho_2(y)dxdy$. Then for $N$ large enough, one has
\begin{gather}
\cH_N(\rho_N | (\rho_1\rho_2)^{\otimes N})\le \frac{C}{N}.
\end{gather}
\end{theorem}

The proof is similar to the proof of Theorem \ref{thm:main}. Here, we list out the key ingredients and omit the details. By noting that
\begin{multline}
(\rho_1\rho_2)^{\otimes N}=(Z_1Z_2)^{-N}
\exp\Bigg(-\beta\Big(\sum_{i=1}^N \left(U_1(x_i)+W_1*\rho_1(x_i)+W_c*\rho_2(x_i)\right)\\
+\sum_{j=1}^N\left(U_2(y_j)+W_2*\rho_2(y_j)+W_c*\rho_1(y_j)\right)\Big)\Bigg),
\end{multline}
we rewrite the $\rho_N$ in \eqref{eq:multirhoN} as
\begin{multline}
\rho_N=Z_N^{-1}
\exp\Bigg(-\beta\Big(E_N(\mathrm{x}, \mathrm{y}) +\frac{N}{2}\int_{\R^d} W_1*\rho_1(x)\rho_1(x)\,dx  \\
+\frac{N}{2}\int_{\R^d} W_2*\rho_2(y)\rho_2(y)\,dy +N\int_{\R^d}W_c*\rho_2(x)\rho_1(x)\,dx  \Big)\Bigg).
\end{multline}
The difference between $\bar{Z}_N$ and $Z_N$ is a constant depending on $N$ but not on $\mathrm{x}, \mathrm{y}$.

It is similarly computed that
\[
\cH((\rho_1\rho_2)^{\otimes N} | \rho_N)=\log\left(\frac{(Z_1Z_2)^{-N}}{Z_N^{-1}}\right)
=Z_N-N\log(Z_1 Z_2) \ge 0,
\]
and
\begin{multline*}
\cH(\rho_N | (\rho_1\rho_2)^{\otimes N})
=N\log(Z_1 Z_2)-Z_N+\\
 \beta \int_{\R^{Nd}\times \R^{Nd}}(\Phi_1(\mathrm{x},\mathrm{y})
+\Phi_2(\mathrm{x}, \mathrm{y})+\Phi_c(\mathrm{x},\mathrm{y})) \rho_N(d\mathrm{x}, d\mathrm{y}),
\end{multline*}
where
\begin{gather*}
\begin{split}
& \Phi_1(\mathrm{x}, \mathrm{y})
=\sum_{i=1}^N W_1*\rho_1(x_i)-\frac{1}{2(N-1)}\sum_{i\neq j}W_1(x_i, x_j)
-\frac{N}{2}\int W_1*\rho_1\rho_1(x)\,dx,\\
&  \Phi_2(\mathrm{x}, \mathrm{y})
=\sum_{j=1}^N W_2*\rho_2(y_j)-\frac{1}{2(N-1)}\sum_{i\neq j}W_2(y_i, y_j)
-\frac{N}{2}\int W_2*\rho_2\rho_2(y)\,dy,  \\
& \Phi_c(\mathrm{x}, \mathrm{y})
=\sum_{i=1}^N (W_c*\rho_2(x_i)+W_c*\rho_1(y_i))
-\frac{1}{N}\sum_{i,j}W_c(x_i, y_j)-N\int W_c*\rho_2 \rho_1 dx.
\end{split}
\end{gather*}

The treatment of $\Phi_1$ and $\Phi_2$ are the same as in Lemma \ref{lmm:expcontrol}. In particular, using Lemma \ref{lem:integralcontrol}, one has
\begin{multline*}
\beta \int_{\R^{Nd}\times \R^{Nd}}(\Phi_1(\mathrm{x},\mathrm{y})
+\Phi_2(\mathrm{x}, \mathrm{y}))\rho_N(d\mathrm{x}, d\mathrm{y})
\le \frac{1}{c}\cH(\rho_N | (\rho_1\rho_2)^{\otimes N}) \\
+\frac{1}{c}\log\left(\int \exp(c\beta(\Phi_1+\Phi_2)) (\rho_1\rho_2)^{\otimes N}   \right).
\end{multline*}
The integral of the exponential can be split into two integrals as in Lemma \ref{lmm:expcontrol}. Hence, the conditions are similar for this to be finite.

The cross term $\Phi_c$ is much simpler. In fact, we do not need the symmetrization technique in \eqref{eqn:phi}. One sets
\[
\phi_c(x_i, y_j)=W_c*\rho_c(x_i)+W_c*\rho_1(y_j)-W_c(x_i, y_j)
-\int_{\R^d} W_c*\rho_2 \rho_1\,dx
\]
such that $\sum_{i, j}\phi_c(x_i, y_j)=N\Phi_c$. One finds that $\phi_c$ satisfies
the required cancellation conditions.
Repeating the proof of Lemma \ref{lmm:expcontrol} with
\[
B_j^k=\phi_c(X_j, Y_k)+\phi_c(X_k, Y_j),
\]
one finds the following fact.
\begin{lemma}
Suppose that $W_c\in L^{\infty}$
and for $\eta<1/(4\sqrt{2}e \|W_c\|_{\infty})$, one has
\[
\sup_{N\ge2}\int \exp\left(\eta \frac{1}{N}\sum_{i, j}\phi_c(x_i, y_j) \right) 
(\rho_1\rho_2)^{\otimes N}(d\mathrm{x}, d\mathrm{y})<\infty.
\]
\end{lemma}
Here, the $\limsup$ is replaced by $\sup_{N\ge 2}$ because 
there is no discrepancy between $N-1$ and $N$.
Moreover, the coefficient is now $4\sqrt{2}$ instead of $2\sqrt{2}$ because
in construction of the martingale difference $B_j^k$, there are more terms.

The condition \eqref{eq:betamultispecies} is required because we used
\begin{gather*}
\beta \int_{\R^{Nd}\times \R^{Nd}} \Phi_c(\mathrm{x},\mathrm{y}) \rho_N(d\mathrm{x}, d\mathrm{y})
\le \frac{1}{c'}\cH(\rho_N | (\rho_1\rho_2)^{\otimes N})
+\frac{1}{c'}\log\left(\int \exp(c'\beta \Phi_c) (\rho_1\rho_2)^{\otimes N}   \right)
\end{gather*}
and need $1/c+1/c'<1$.

\section{Numerical experiments}
\label{sec:numer}

According to Theorem \ref{thm:main}, Theorem \ref{thm:convmulti} and Corollary \ref{cor:sobolevestimate}, both the one marginal of the $N$-body Gibbs measure and the empirical measure constructed from the samples converge to the solution of the stationary Mckean-Vlasov equation. 
This then justifies our motivation to use sampling methods to solve the stationary nonlinear Fokker-Planck equations. 
Upon adequately sampling from the Gibbs measure \eqref{eqn:Gibbsnbody}, one can obtain a good approximation to the solution of the stationary Mckean-Vlasov equation.

To sample from the Gibbs distribution like \eqref{eqn:Gibbsnbody}, standard Markov Chain Monte Carlo method, like the Langevin MC \cite{roberts1996exponential} or the direct Metropolis-Hastings (MH) algorithm \cite{metropolis1953equation} are not convenient as $W$ could be singular and the complexity for moving one particle is $\mathcal{O}(N)$. The random batch Monte Carlo (RBMC) method proposed in \cite{li2020random} can overcome these two difficulty, where the singularity is treated by a kernel splitting strategy while the complexity is reduced to $\mathcal{O}(1)$ by a random batch strategy \cite{jin2020random}.

Let us briefly describe RBMC here. We decompose the singular interacting potential $W$ into \(W=W_1+W_2\), where $W_2$ is singular and of short range while $W_1$ is the remaining part which is is smooth and often of long range. Then, $W_1$ is used in the Langevin dynamics step, while $W_2$ is used in the Metropolis rejection step. Regarding the complexity, consider only the interactions within a randomly selected mini batch, the Langevin step is $\mathcal{O}(1)$.
Besides, Metropolis rejection step is also practically $\mathcal{O}(1)$ since $W_2$ is local. The RBMC method is summarized in Algorithm \ref{algo: RBMC}.
\begin{algorithm}[H]
\caption{Random Batch Monte Carlo method}
\begin{algorithmic}[1]
\STATE Split $W=W_1+W_2$. Generate $N$ initial particles $X_i, i=1,\cdots,N$. Choose $m\geq 1$, batch size $p>1$, number of iterations in the burn in phase and sampling phase $N_b, N_s$.
\FOR{$n=1:N_b+N_s $}
\STATE Randomly select a particle $X_i$ with uniform probability.
 \FOR{$k=1:m$}
 \STATE Choose $\xi_{k}$ (random set not containing $i$), $z_k\sim N(0,I_d)$, $\tau_k>0$.
	$$
	X_i=X_i-\left(\nabla U(X_i)+\frac{1}{p-1}\sum_{j\in \xi_k}\nabla W_1 (X_i-X_j)\right) \tau_k +\sqrt{\frac{2\tau_k}{\beta}} z_k.
	$$
 \ENDFOR
	\STATE Compute the acceptance ratio using cell list or other data structures:
	$$
	a(X_i^*,X_i) =\min\left\{1,\exp\left(-\beta\sum_{j:j\neq i}\left( W_2 (X_i^*-X_j)- W_2 (X_i-X_j)\right)\right) \right\}.
	$$
 Accept $X_i^*$ with probability $a(X_i^*,X_i)$.
 
 \ENDFOR
\end{algorithmic}
\label{algo: RBMC}
\end{algorithm}

In this section, we will first consider the Poisson-Boltzmann equation where two types of ions are considered. Note that the interaction kernel is not bounded so the theory will not apply but the numerical results still agree with the theory. As a second example, we look at using sampling for training the two-layer neural networks.

\subsection{The Poisson-Boltzmann equation}\label{subsec:pb}

The PB equation \cite{guoy1910constitution,chapman1913li} is a classical implicit solvent model describing the electric potential at equilibrium, widely used in electrochemistry. Consider a charged colloid $\Omega$ immersed in some symmetric monovalent electrolyte. Then the ions (two types with valences $z_+=1, z_-=-1$) will concentrate close to the surface of $\Omega$ and form a screening layer. At equilibrium, the charge distribution outside the charged surface is given by the Boltzmann distribution, i.e., 
$$\rho_+\propto \exp(- \beta \phi),\quad \rho_-\propto \exp(\beta \phi).$$ Here $\phi$ denotes the electric potential.
The concentration of both anions and cations are $\rho_{\infty}$ in the far field. The potential then satisfies
\[
-\varepsilon \Delta \phi = 
\begin{cases}
\rho_{\infty} (e^{-\beta \phi}-e^{\beta \phi}), & x\in \Omega^c,\\
\rho_f, & x\in \Omega.
\end{cases}
\]
Here $\varepsilon$ is the dielectric constant. The equation in $\Omega^c$ is the PB equation. 

To do simulation, we truncate the external solution into bounded domain $D_L=B(0,L)\backslash \Omega$ and prescribe a Neumann boundary condition 
$$
-\frac{\partial \phi}{\partial n} = 0, \quad x\in \partial B(0,L).
$$
It has been shown that this approximation problem converges to the PB equation for $L$ large enough \cite{li2022some}.

To solve the PB equation by sampling method, we consider the corresponding particle system where the interaction potential between the particles is the Coulomb potential. 
Given the free charge $\rho_f$, we adjust the total charge of the cations $Q_+$ so that the total charge of the anions satisfy
\[
Q_+-Q_-+\int \rho_f dx=0.
\]
We choose a numerical charge unit $q$ such that the numbers of numerical cations and anions are given respectively by
\[
N_+=\lceil \frac{Q_+}{z_+q}\rceil,\quad N_-=\lceil\frac{Q_-}{|z_-|q}\rceil.
\] 
Note that here $N_+\neq N_-$ but if we decrease $q$, $N_+$ and $N_-$ will increase proportionally so that the proof in section \ref{sec:multispecies} can be applied similarly.  Let $W$ be the fundamental solution to 
\[
-\varepsilon \Delta W=\delta.
\]
Then the external potential is given by $U = \rho_f*W$ is the external potential and the Gibbs distribution can be written as
\begin{gather}
\rho_N\propto \exp\left(-\beta(\sum_i z_i U(x_i)+\frac{1}{2}\sum_{i\neq j} z_i z_j q W(x_i-x_j))\right)1_{\{x_i\in D_L, 1\le i\le N\}},
\end{gather}
where $z_i$ is the valence of $i$-th particle. Note that the Hamiltonian used 
is scaled by $1/q$ so that the mean field limit exists, consistent with Remark \ref{rmk:rescaledHam}. In fact, if we consider the corresponding Langevin equation, one has
\[
dX_i=-z_i(\nabla U(x_i)+\sum_{j\neq i}qz_j \nabla W(x_i-x_j))\,dt+\sqrt{2\beta^{-1}}dW_i+dR_i,
\]
where $dR_i$ is the reflection at the boundary (see \cite{li2022some}). The formal mean field limit is 
\[
\begin{split}
\partial_t\rho_+ &=\nabla\cdot(\rho_+(\nabla U(x)+\nabla W*\rho_+-\nabla W*\rho_-))+\beta^{-1}\Delta\rho_+,\quad x\in D_L\\
\quad \partial_t\rho_-&=-\nabla\cdot(\rho_-(\nabla U(x)+\nabla W*\rho_+-\nabla W*\rho_-))+\beta^{-1}\Delta\rho_-, \quad x\in D_L.
\end{split}
\]
This is the system of PNP equations, with the limit to be the PB equation. 

For 1D case, $W(x)=(2\varepsilon)^{-1}|x|$, and $\nabla W$ is bounded, the theory above can be applied. In fact, though the domain here is bounded, the proof can be performed without difficulty. 
For 3D case, $W(x)=\frac{1}{4\pi \varepsilon |x|}$, the theory above do not apply. If we use this potential directly in the Gibbs distribution, it is not integrable due to the singularity between one anion particle and one cation particle. One is adding some hard sphere repulsive potential to avoid this attraction and cancellation. For example, the Lennard-Jones potential \cite{jones1924determination} given by
\begin{equation}
	\label{eq:Lennard-Jones}
	U_{\mathrm{LJ}}(r)=4\epsilon\left[\left(\frac{\sigma}{r}\right)^{12}-\left(\frac{\sigma}{r}\right)^6\right],
\end{equation}
where $\epsilon$ is the depth of the potential well (i.e. the strength of Lennard-Jones potential), $\sigma$ is the finite distance at which the inter-particle potential is zero. The parameters $\epsilon, \sigma$ vanish in the mean field limit so that the effects of the Lennard-Jones does not appear in the PB equation.
Another approach is to take a cutoff of $W$ that vanishes as $N \to \infty$. For example, 
$$
\bar{W}(r)=\left\{
\begin{aligned}
     & \frac{1}{8\pi \varepsilon r_N}\left(3-\frac{r^2}{r_N^2}\right), \quad & 0<r<r_N,\\
   & \frac{1}{4\pi \varepsilon r},\quad  & r>r_N,
\end{aligned}
\right.$$
where $r_N=N^{-\gamma(d)}$ for some $\gamma(d)>0$. Clearly, $|\bar{W}| \leq \frac{3}{8\pi \varepsilon r_N}=\frac{3}{8\pi \varepsilon } N^{\gamma}$. $\bar{W}$ converges to $W$ as $N\to \infty$.
Either with the hard sphere potential or the cutoff, the Gibbs measure in 3D then can be normalized to a probability measure. Though our theory does not apply due to the singularity, we can still test the performance numerically nevertheless.

For 1D case, the particles can be imagined as charged sheets in 3D and the Green's function does not blow up at $x=0$. there is no need to do splitting in the RBMC. That means we run a particle using the random minibatch version of a Langevin equation and then accept the proposal totally. This is in fact a modified version of the stochastic gradient Langevin dynamics (SGLD) \cite{welling2011bayesian}, but we are allowed to move several particles only in each iteration. When the particles fall out of the domain considered, we do simple reflection.
\begin{figure}
    \centering
    \includegraphics[width=\textwidth]{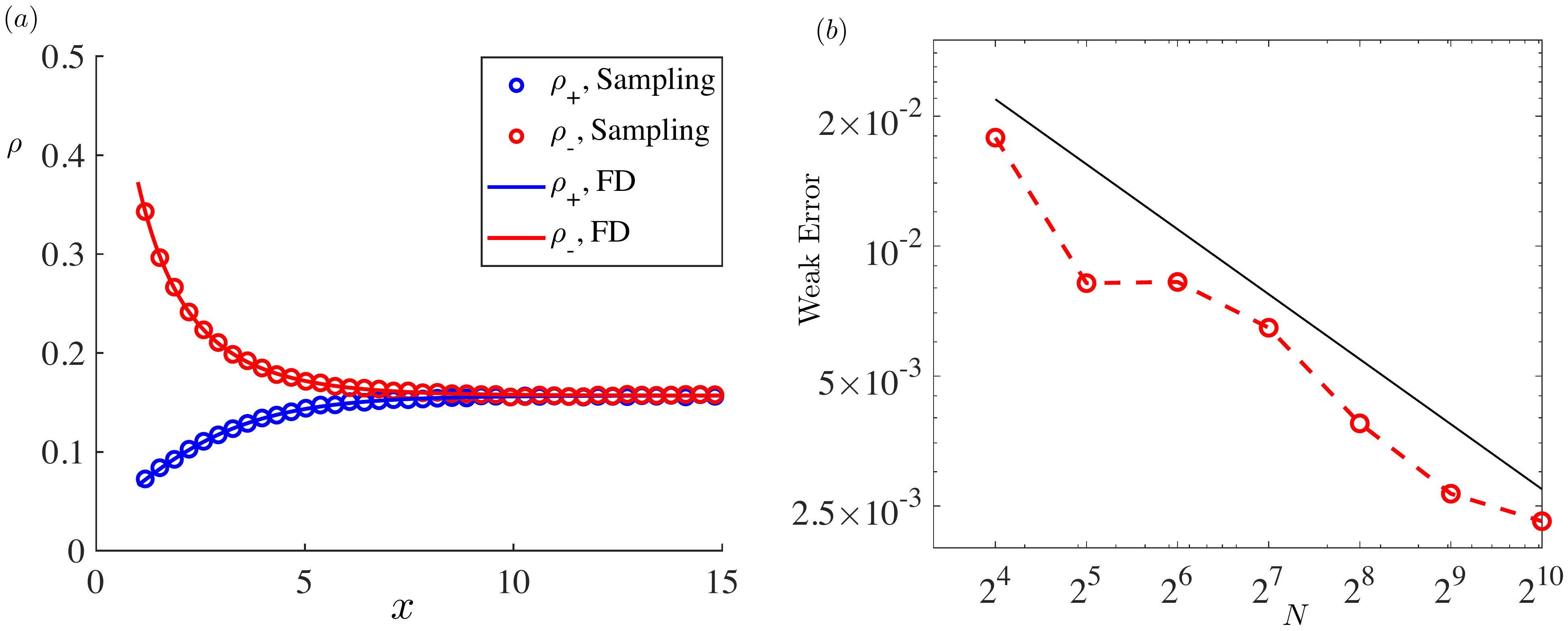}
    \caption{Solving the 1D PB equation using sampling. (a). The density distribution obtained by RBMC sampling method versus finite difference method. The parameters are detailed in the context. (b) the mean square weak error versus particle number for positive charge (number for negative charge is scaled proportionally). The $1/2$ rate is observed.}
    \label{fig:1d}
\end{figure}

In Figure \ref{fig:1d} (a), we present the simulation results for solving the 1d PB equation with sampling. The particles are all put in $(1, L)$ with $L=15$ in the experiment and a free charge $Q_f$ is put at $x=0$ (for 1D case, the location of the free charge in $(-\infty, 1)$ is not important).   Here, $\varepsilon=1$, $Q_f=0.5$, $Q_+=2$, $Q_-= 2.5$, and  $N_+=1024$ so that $q=Q_+/N_+$.  We apply  RMBC without splitting for sampling. Since there is no rejection, we move the $N$ particles. The batch size is $p=4$, step size $\tau=0.005$, burn in time $T_b=500$ (i.e, after $T_b/\tau$ steps) and stop time $T_e=2000$.  

In Figure \ref{fig:1d} (b), we show the relative weak error of the sampling method versus the particle number of cations $N_+$  used in the Gibbs distribution.
In one experiment, for each given $N_+$, we collect $N_s=4.8\times 10^6$ samples (i.e., $4.8\times 10^6/N$ configurations) to compute 
\begin{equation}
\mathrm{err}:=\left|\int{f(x)\rho(x)\,dx}-\frac{1}{N_s}\sum_{i=1}^{N_s}{f(x_i)}\right| / \int{f(x)\rho(x)\,dx}.
\end{equation}
Such a quantity is used to approximate $\E \int \varphi d\mu_N$ for the empirical measure $\mu_N$.  The reference solution $\rho$ is computed by the finite difference method as explained in \cite{li2022some}. The test function used is $f(x)=x^2$ and the errors for the positive and negative particles are computed separately and we repeat the experiments for $M=32$ times to get the error mean square weak error (MSWE)
\begin{equation}
\mathrm{MSWE} = \sqrt{\frac{\sum \mathrm{err}_+^2 + \sum \mathrm{err}_-^2}{2M}}.
\end{equation}
The results are shown for $N_+=16, 32, 64, 128, 256, 512, 1024$. Clearly, the convergence rate is $1/2$ which is consistent with the result $\|\mu_N-\mu\|_{H^{-s}}\sim N^{-1/2}$.

For the 3D case, to run the RBMC algorithm, we let $r=|x|$ and split the singular Coulomb potential as
$$
W_1(r)=\left\{
\begin{aligned}
	&\frac{1}{4\pi\varepsilon}\left[-\frac{1}{r_c^2}(r-r_c)+\frac{1}{r_c}\right],\quad &0<r\leq r_c,\\
	&\frac{1}{4\pi\varepsilon}\frac{1}{r}, & r>r_c,
\end{aligned}
\right.
$$
and
$$
W_2(r)=\left\{
\begin{aligned}
	&\frac{1}{4\pi\varepsilon}\left[\frac{1}{r}+\frac{1}{r_c^2}(r-r_c)-\frac{1}{r_c}\right],\quad & 0<r\leq r_c,\\
	& 0,& r>r_c.
\end{aligned}
\right.
$$
\begin{figure}
    \centering
    \includegraphics[width=7cm]{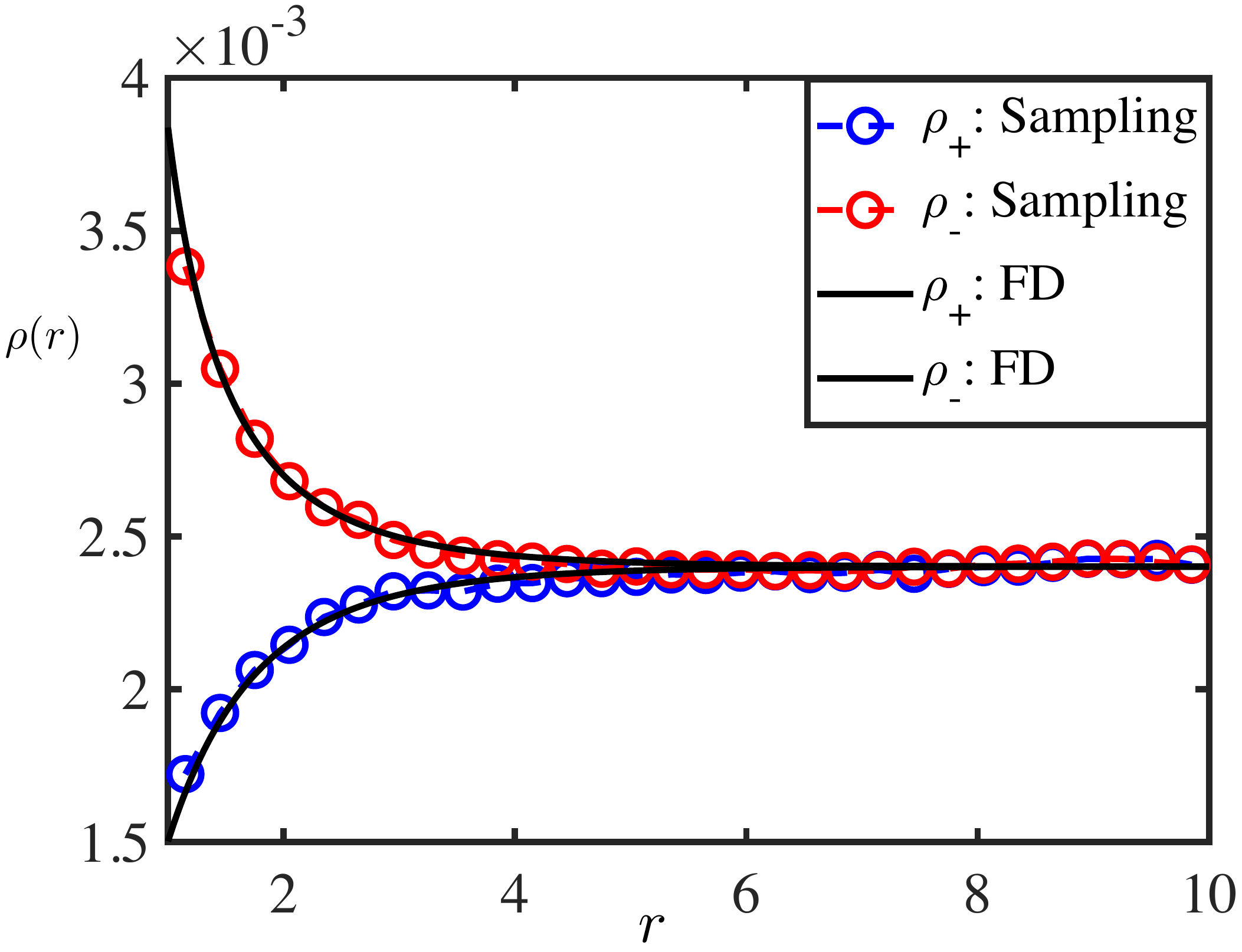}
    \caption{Solving 3D PB equation by RBMC sampling. The LJ potential is used to overcome the singularity of 3D Coulomb potential. }
    \label{fig:3d}
\end{figure}
In Figure \ref{fig:3d}, we present the simulation results for $\varepsilon=0.01$, $Q_f=10\varepsilon$ centered at the origin, $Q_+=10$. Here, $\rho_{\pm}(r)$ are the 3D densities such that $\int \rho_+(r)4\pi r^2 dr=Q_+$. The cell is $B(0, 1)$ and $Q_f$ is put at the origin. The domain is truncated at $L=10$. In the RBMC sampling method, we use $N_+=100$ particles for cations so that $q=0.1$. The cut-off radius of the Coulomb potential is chosen as $r_c=0.1$, while the zero and strength of LJ potential is $\sigma=r_c=0.1$ and $\epsilon=0.01$ respectively. We move one particle each step with batch size $p=2$, steps per iteration $m=9$,  step size $\tau=0.01$. Also, when particles fall out of the domain, a simple reflection is performed.
The numbers of iterations in the burn-in phase and the sampling phase are $N_b=5\times 10^5$ and $N_s=5\times 10^7$. To compute the density distributions at equilibrium, we collect the samplers every $1000$ iterations in the sampling phase. Hence, the samplers used for cations and anions is $N_s*N_+/1000=5\times 10^6$ and $N_s*N_-/1000=5.05\times 10^6$. 
Clearly, we can see from Fig. \ref{fig:3d} that the RBMC sampling result matches well with the reference solution, given by the finite difference method as mentioned in \cite{li2022some}.

\subsection{Two layer neural networks}

Following \cite{mei2018mean,rotskoff2018trainability}, we consider in this subsection the supervised learning task to fitting data of the form $\{(x, y)\}$ using a two-layer neural network of the following form
\begin{equation}\label{eq:twonn}
\hat{y}(x;\boldsymbol\theta)=\frac{1}{N}\sum_{i=1}^N\sigma_*(x;\theta_i).
\end{equation}
Here, $\hat{y}(x;\boldsymbol\theta)$ represents the output of the network for input $x$, and $\sigma_*(x;\theta_i)$ is a neuron with parameter $\theta_i$. Often 
\begin{gather}\label{eq:sigmastar}
\sigma_*(x,\theta)=c\sigma(w^Tx+b), \quad \theta=(c, w, b),
\end{gather}
and $\boldsymbol\theta=(\theta_1,\theta_2,\cdots, \theta_N)$. Hence, the model \eqref{eq:twonn} is a two-layer neural network with the hidden layer to be $x\mapsto w^Tx+b$, and the output layer is $(\sigma_i)\mapsto \frac{1}{N}\sum_i c_i \sigma_i$ after the activation. The mean square loss function is given by
\begin{gather}\label{eq:nnmse}
R(\boldsymbol\theta)=\frac{1}{2}\mathbb{E}_{x,y}|\hat{y}(x;\boldsymbol\theta)- y|^2.
\end{gather}
If we are given $P$ training data $\{(x_i, y_i)\}$, then
\begin{gather}\label{eq:nnempmse}
\frac{1}{2}\mathbb{E}_{x,y}|\hat{y}(x;\boldsymbol\theta)- y|^2
=\frac{1}{2P}\sum_{j=1}^P |\hat{y}(x_j;\boldsymbol\theta)- y_j|^2.
\end{gather}
If we have the task for fitting the distribution of the data, the expectation is then taken over the population distribution, which is the law of $(x, y)$ as $P\to\infty$.
We will use $\E$ to mean the expectation over $(x, y)$, either over the training data or over the population distribution.

An important observation in \cite{mei2018mean,rotskoff2018trainability} is that the loss function can also be written as
\begin{gather}
R(\boldsymbol\theta)=\frac{1}{2 N^2}\sum_{i, i'}W(\theta_i, \theta_i')
+\frac{1}{N}\sum_{i=1}^N U(\theta_i)+\frac{1}{2}\E|y|^2
=: E(\mu_N)+C, 
\end{gather}
where
\begin{gather}
E(\mu):=\int U \mu(d\theta)+\frac{1}{2}\int W(\theta, \theta')\mu(d\theta)\mu(d\theta'),
\quad 
\mu_N=\frac{1}{N}\sum_{i=1}^N \delta(\cdot-\theta_i),
\end{gather}
and $U$ and $W$ are given respectively by
\begin{equation}
U(\theta) =-\mathbb{E}[y\sigma_*(x,\theta)],\quad
W(\theta,\theta')=\mathbb{E}[\sigma_*(x;\theta)\sigma_*(x;\theta')].
\end{equation}
This indicates that the loss function is in fact an energy functional of the empirical measure of the parameters of the neurons.

The noisy stochastic gradient descent (SGD) algorithm is often used to train the neuron networks \cite{mei2018mean}
\begin{equation}\label{eq:noisysgd}
\theta_i^{k+1}=\theta_i^k-\lambda s_k\theta_i^k+s_k(y_k-\hat{y}_k)\nabla_{\theta_i}\sigma_*(x_k;\theta_i^k)+\sqrt{2s_k/\beta}z_i^k,
\end{equation}
where $z_i^k\sim\mathcal{N}(0, I_D)$ where $D$ is the dimension of $\theta_i$, $\hat{y}_k=\hat{y}(x;\boldsymbol{\theta^k})$. Here, $\lambda\theta_i^k$ comes from an $\ell^2$-regularization \cite{mei2018mean}. Note that the coefficients used here are slightly different from there because we think the current version is more natural. 
This update rule is clearly an approximation of 
\[
\theta_i^{k+1}=\theta_i^k-\lambda s_k\theta_i^k-s_k \nabla_{\theta_i}(NR(\boldsymbol{\theta^k}))+\sqrt{2s_k/\beta}z_i^k,
\]
where only one sample $(x_k, y_k)$ is used to replace the expectation over the distribution of the data in $R(\theta)$, and this is the mini-batch idea in SGD. Of course, several samples could be used. 
Note that the rescaled loss (or energy functional) $NR(\boldsymbol{\theta^k})$ is used in consistent with Remark \ref{rmk:rescaledHam}. This is the discretization of 
\[
d\theta_i=-\lambda \theta_i-\nabla_{\theta}(NR(\boldsymbol{\theta}))\,dt+
\sqrt{2\beta^{-1}}\,dB_i.
\]
Clearly the mean field limit of this interacting particle system would be the following nonlinear Fokker-Planck equation
\begin{gather}\label{eq:nonfpnn}
\partial_t\rho=\nabla_{\theta}\cdot(\rho \nabla \Psi_{\lambda})+\beta^{-2}\Delta\rho,
\end{gather}
where
\[
\Psi_{\lambda}=U(\theta)+\int W(\theta, \theta')\rho(d\theta')
+\frac{\lambda}{2}|\theta|^2=\frac{\delta}{\delta\rho}\left(E(\rho)+\lambda\int \frac{|\theta|^2}{2}\rho(d\theta)\right).
\]
This mean field limit has been justified in a probabilistic sense in \cite{mei2018mean}. Clearly, for the $N$-neuron system, the stationary distribution is 
\begin{gather}\label{eq:gibbsnn}
\rho_N\propto \exp\left(-\beta \Big(\sum_i (U(\theta_i)+\frac{\lambda}{2}|\theta_i|^2)+\frac{1}{N}\sum_{i, j}W(\theta_i, \theta_j)\Big)\right).
\end{gather}
As proved in \cite{mei2018mean}, the stationary solution of the nonlinear Fokker-Planck equation \eqref{eq:nonfpnn} is close to the global minimizer of the two layer neural networks with infinite width, and has nearly zero training loss.

Although our theory above is for $W(x, y)=W(x-y)$, the proof can be carried here without difficulty for general symmetric $W$. We thus conclude that
\begin{corollary}
Suppose that $W$ is bounded under suitable choices of activation functions and the the domains for $\theta$. When $\beta$ is not very big, as the width of neural network $N\to\infty$, the empirical measure $\mu_N$ converges to the stationary solution $\rho$ in $H^{-s}$ for $s>D/2$ almost surely. Moreover, the neural network converges to the nearly optimal predictor $\int \sigma_*(x, \theta)\rho(d\theta)$ for all $x$.
\end{corollary}
The second part is actually a corollary of Corollary \ref{cor:boundness}.
In practice, training a very wide network requires more memory. 
Moreover, it is challenging to reach a true ``steady state" during training. The above result then suggests that we may do sampling from the $N$-body Gibbs measure to approximate the nearly optimal predictor.

\begin{figure}
    \centering
    \includegraphics[width=0.9\textwidth]{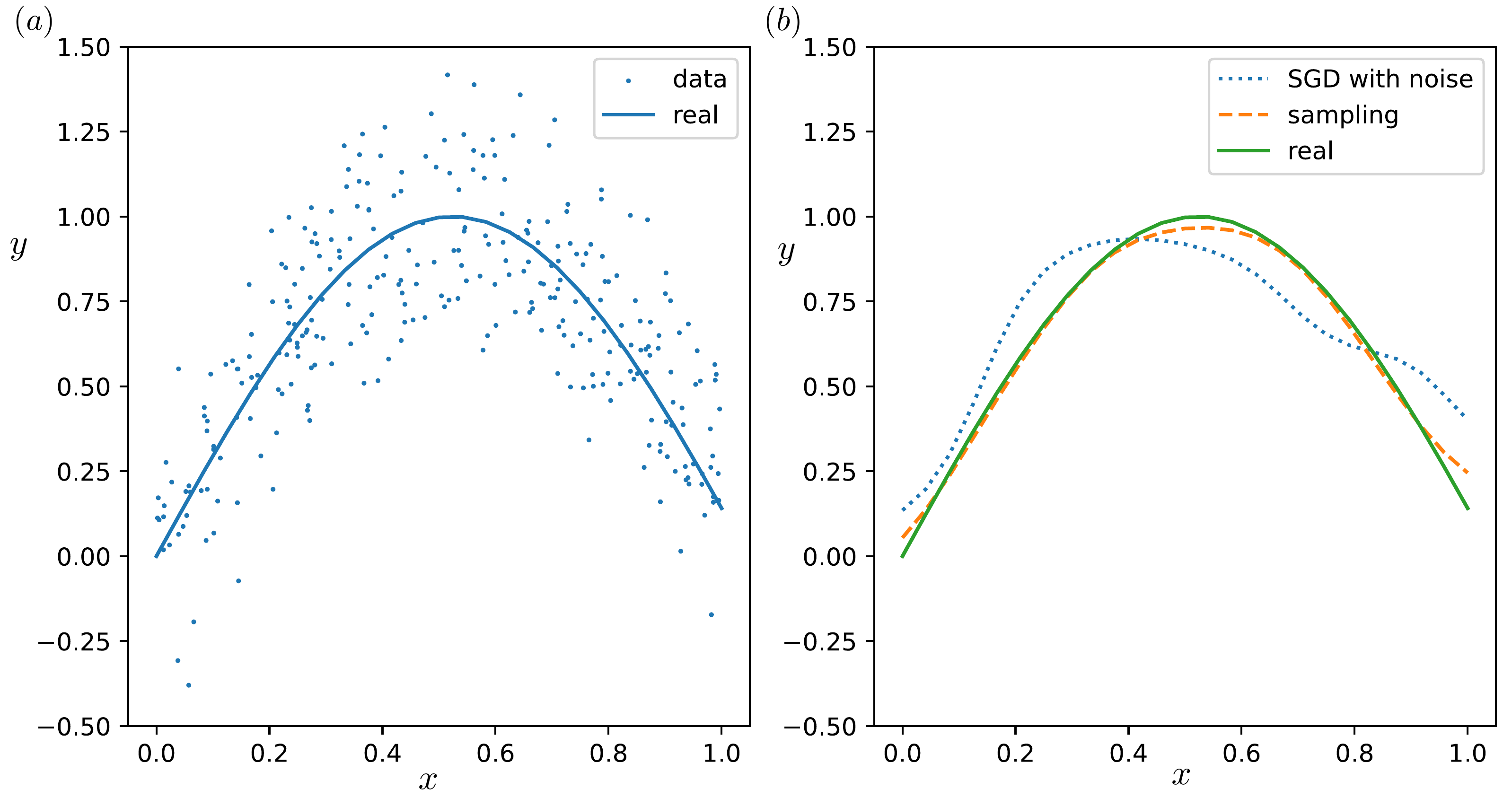}
    \caption{Training the two-layer neural networks by samplig.}
    \label{fig:nn}
\end{figure}

Let us consider a typical regression task. Here, measurement of the output value is taken after each input, with measurement errors following a Gaussian distribution. The goal is to predict new outputs for a new input. Consider the following data which are generated by
\begin{equation}
y_i=\sin(3x)+\varepsilon_i, \varepsilon_i\sim\mathcal{N}(0, 0.2), x\sim U[0,1].
\end{equation}
We use the two layer neural network model above to fit the data to approximate the model $y=f(x)=\sin(3x)$. Here, we use the sigmoid function $\sigma(x)=\frac{1}{1+e^{-x}}$ in \eqref{eq:sigmastar}.
 The network width is chosen as $N=64$.

The noisy SGD update \eqref{eq:noisysgd} is applied with  $\lambda=0$ for optimization. The time step (or learning rate) $s_k=10$ and $\beta=2000$.
For sampling, we apply RBMC without splitting. Hence, the algorithm becomes almost the same with SGLD or the noisy SGD, with the difference that we are allowed to update a few neurons in each iteration. The stepsize is the same as above. 
In the RBMC, a warm-up phase of $N_b=10000$, a total number of iterations of $N_s=20000$. In other words, we take $N\times N_s = 1.28\times 10^6$ samples for the empirical distribution of $\theta$. The steady-state distribution is reconstructed using the empirical distribution:
\begin{equation}
\tilde{y}(x;\boldsymbol{\theta})=\mathbb{E}[\sigma_*(x;\boldsymbol{\theta})]\approx\frac{1}{M}\sum_{i=1}^M\sigma_*(x;\theta_i).
\end{equation}

Figure \ref{fig:nn}  illustrates the numerical results. The results by sampling (dashed line) is closer to the true curve $y=f(x)$ compared to the SGD training results (dotted line). It exhibits better training and generalization errors. This is clearly expected as we used more samples for $\theta$ and the results by sampling can be thought as the average of $N_s$ networks with width $N=64$. 

\begin{table}[h]
\centering
\caption{Neural Network Errors}\label{tab:nnloss}
\begin{tabular}{ccc}
\toprule
    &$\overline{\mathrm{err}}_{sgd}$&$\mathrm{err}_{s}$\\
\midrule
Training Set&0.0478&0.0409\\
Test Set&0.0490&0.0420\\
\bottomrule
\end{tabular}
\end{table}

Table \ref{tab:nnloss} shows the errors of the neural network. The "$\overline{\mathrm{err}}_{sgd}$" refers to the average of $N_s=2\times 10^4$ runs of the SGD mean square errors defined in \eqref{eq:nnempmse}, while ``$\mathrm{err}_{s}$'' corresponds to the mean square error \eqref{eq:nnempmse} of the empirical distribution. Clearly, the errors of the predictor constructed by sampling is smaller.

\section*{Acknowledgement}

This work was financially supported by the National Key R\&D Program of China, Project Number 2021YFA1002800 and 2020YFA0712000. The work of L. Li was partially supported by NSFC 12371400 and 12031013,  Shanghai Municipal Science and Technology Major Project 2021SHZDZX0102, Shanghai Science and Technology Commission (Grant No. 21JC1403700, 20JC144100, 21JC1402900), and the Strategic Priority Research Program of Chinese Academy of Sciences, Grant No. XDA25010403.

\bibliographystyle{plain}
\bibliography{meanfield}

\end{document}